\documentclass[leqno]{aadmbook}
\usepackage{amsthm}
\usepackage{amsfonts}
\usepackage{amsmath}

\theoremstyle{plain}
\newtheorem{theorem}{Theorem}

\newtheorem{lemma}{Lemma}
\newtheorem{corollary}{Corollary}
\newtheorem{proposition}{Proposition}

{\itshape}{\rmfamily}
\newtheorem{remark}{Remark}

\usepackage{amsthm,amsmath,amssymb}
\usepackage{graphicx}
\usepackage[colorlinks=true,citecolor=black,linkcolor=black,urlcolor=blue]{hyperref}
\usepackage[english]{babel}
\usepackage{amsfonts}
\usepackage{epsfig, subfigure}
\usepackage{amscd,latexsym}
\usepackage{url}
\usepackage{color}
\usepackage{caption}

  \addto{\captionsenglish}{\renewcommand{\bibname}{\large \bf References}}

\begin{document}

% Your \newcommands below (if there are any):

\oddsidemargin 16.5mm
\evensidemargin 16.5mm

\thispagestyle{plain}

%\begin{center}
%{\large \sc  Applicable Analysis and Discrete Mathematics}

%{\small available online at  http:/$\!$/pefmath.etf.rs }
%\end{center}

%\noindent{\small{\sc  Appl. Anal. Discrete Math.\ }{\bf x} (xxxx),
%xxx--xxx.} \hfill{\scriptsize doi:10.2298/AADMxxxxxxxx}

\vspace{5cc}
\begin{center}

{\large\bf  DOMINATING $2$-BROADCAST IN GRAPHS: COMPLEXITY, BOUNDS AND EXTREMAL GRAPHS
\rule{0mm}{6mm}\renewcommand{\thefootnote}{}
\footnotetext{\scriptsize 
E-mail: \texttt{\{jcaceres,mpuertas\}@ual.es}, \texttt{\{carmen.hernando,merce.mora,ignacio.m.pelayo\}@upc.edu}
}}

\vspace{1cc}
{\large\it J. C\'aceres$^1$, C. Hernando$^2$, M. Mora$^2$, I.M. Pelayo$^2$, M.L. Puertas$^1$}
\end{center}
{\it $^1$ Department of Mathematics, Universidad de Almer\'ia, Spain\\
$^2$ Department of Mathematics, Universitat Polit\`{e}cnica de Catalunya, Spain}

\begin{center}
\vspace{1cc}
\parbox{24cc}{{

\small
\vspace{1.5cc}
{\bf Abstract}
\\
Limited dominating broadcasts were proposed as a variant of dominating broadcasts, where the broadcast function is upper bounded. As a natural extension of domination, we consider dominating $2$-broadcasts along with the associated parameter, the dominating $2$-broadcast number. We prove that computing the dominating $2$-broadcast number is a NP-complete problem, but can be achieved in linear time for trees. We also give an upper bound for this parameter, that is tight for graphs as large as desired.
\vspace{1cc}

\noindent {\bf 2010 Mathematics Subject Classification.} 05C12, 05C69.\\
{\bf Keywords.} Broadcast. Domination. NP-complete decision problem.
}}

\end{center}

\vspace{1cc}

\vspace{1.5cc}
\begin{center}
{\bf 1. INTRODUCTION}
\end{center}

Graphs are a widely used tool to represent and study networks, understood in a broad sense, from road or electric networks to networks of distribution or social networks. One of the fundamental problems that has traditionally been studied in this area is the optimization of resources in a network and the properties of domination in graphs have played an essential role in the modeling and resolution of this type of issues, since this concept was originally defined in the late fifties \cite{Berge58} and named in the early sixties \cite{Ore62}.

A dominating set represents the places in the network where a resource must be placed to which all points of the network must have access and a dominant set of minimum cardinal is the optimal way to distribute that resource along the network. It is well known that there are multiple variants of classical domination, which refer to some particular aspect that one wishes to highlight in a special way. Regarding broadcast domination introduced in \cite{Liu68,Erwin04}, the idea behind this concept is to model several broadcast stations, with associated transmission powers, that can broadcast messages to places at distance greater than one. The original definition does not take into account any limitation of the power of such transmission, that is, it assumes that any point of the network can be reached with a transmission from any origin. This assumption may not be entirely accurate in the case of very large networks or simply due to the use of limited power transmitters. These restrictions naturally give rise to the limited version of the dominating broadcast, in which each transmitter, regardless of where it is placed, has a fixed transmission range. This idea was already suggested in \cite{DEHHH06}, as a possible extension of the original definition of dominating broadcast, although it is not studied yet.

We devote this paper to deepening this concept of limited broadcast domination, in case the transmitting power is limited by $2$, pointing out fundamental differences with the non-limited version. It is known that the computation of the domination broadcast number if polynomial in general graphs (see~\cite{HL06}), however the limited version does not share this property but follow the behaviour of other domination-type parameters, such as the domination number itself: it is an NP-complete problem in general graphs although it is linear in trees. After proving that the computation of this parameter is not easy, our main objective is to determine an upper bound of the minimum number of transmitters, with power limited by $2$, that we need to cover an arbitrary network.

The formal definition of dominating broadcast, taken from \cite{Erwin04}, is as follows: For a graph connected $G$ any function $f\colon V(G)\to \{0,1,\dots ,diam(G)\}$, where $diam(G)$ denotes the diameter of $G$, is called a \emph{broadcast} on $G$.
A vertex $v\in V(G)$ with $f(v)>0$ is a \emph{f-dominating vertex} and it is said to \emph{f-dominate} every vertex $u$ with $d(u,v)\leq f(v)$.
A \emph{dominating broadcast} on $G$ is a broadcast $f$ such that every vertex in $G$ is $f$-dominated and the \emph{cost} of $f$ is $\omega(f)=\sum_{v\in V(G)} f(v)$. Finally, the \emph{dominating broadcast number}, denoted as $\gamma_{\stackrel{}{B}}(G)$, is the minimum cost among all the dominating broadcasts in $G$.

A number of issues has been addressed regarding this variation of domination. For instance the role of the dominating broadcast number into the domination chain is discussed in \cite{DEHHH06}, and the characterization of graphs where the dominating broadcast number reaches its natural upper bounds, radius and domination number, is given in \cite{CHM11,HM09,MW13}. However, the most interesting feature about dominating broadcasts is that $\gamma_{\stackrel{}{B}}(G)$ can be computed in polynomial time (see \cite{HL06}) in any graph $G$. This is quite counter-intuitive since computing almost any domination parameter is in NP.

As the classical domination notion can be viewed as a limited dominating broadcast with range one, we follow the suggestion posed in \cite{DEHHH06} as the open problem of considering the broadcast dominating problem with a limited broadcast power of two.
We formally define the concept of dominating $2$-broadcast in Section~2 and we present some basic properties regarding the associated parameter: the $2$-broadcast dominating number. In Section~3, it is proved that computing the dominating $2$-broadcast number is NP-complete in general but can be done in linear time for trees. Section~4 is devoted to obtain a general upper bound for the dominating $2$-broadcast number and the family of trees that attain this bound is characterized. Finally in Section~5, having in mind that the upper bound in general graphs has been obtained using spanning trees, we study the behaviour of the parameter in graphs with very simple spanning trees, such are graphs having a dominating path.

All graphs considered in this paper are finite, undirected, simple and connected.
The \emph{open neighborhood} of a vertex $v$ is $N(v)$ the set of its neighbors and its  \emph{closed neighborhood} is $N[v]=N(v)\cup \{v\}$.
A \emph{leaf} is a vertex of degree one and its unique neighbor is a \emph{support vertex}.
For a pair of vertices $u,v$ the distance $d(u,v)$ is the length of a shortest path between them.
For any graph $G$, the \emph{eccentricity} of a vertex $u\in V(G)$ is
$\max \{ d(u,v) : v\in V(G) \}$ and is denoted by $ecc_G(u)$.  The maximum (resp. minimum) of the eccentricities among all the vertices of $G$ is the \emph{diameter} (resp. \emph{radius}) of $G$, denoted respectively by $diam (G)$ and $rad (G)$. Two vertices $u$ and $v$ are \emph{antipodal} in $G$ if $d(u,v)=diam(G)$.
A \emph{caterpillar} is a tree such that the set of vertices of degree greater than one induces a path.
Given a tree $T$, a vertex in $u\in V(T)$ and an edge $e\in E(T)$, the tree $T(u,e)$ is the subtree containing $u$, obtained from $T$ by deleting the edge $e$.
For further undefined general concepts of Graph Theory, see \cite{CLZ11}.

\vspace{1.5cc}
\begin{center}
{\bf 2. DOMINATING $2$-BROADCAST}
\end{center}

We begin this section with the formal definition of dominating $2$-broadcast, following the ideas proposed in \cite{DEHHH06}.
To our knowledge, this concept has not {previously} been studied, except in \cite{JK13} where some straightforward properties are introduced.

Let $G$ be a graph. For any function $f\colon V(G)\to \{0,1,2\}$, we define the sets $V_f^0=\{u\in V(G)\colon f(u)=0\}$ and $V_f^+=\{v\in V(G)\colon f(v)\geq 1\}$.
We say that $f$ is a \emph{dominating 2-broadcast} if for every $u\in V(G)$ there exists a vertex $v\in V_f^+$ such that {$d(u,v)\leq f(v)$. In such a case, we say that $u$ \emph{hears} $v$.}
The \emph{cost} of a dominating $2$-broadcast $f$ is $\omega(f)=\sum_{u\in V(G)} f(u)=\sum_{v\in V_f^+} f(v)$.
Finally, the \emph{dominating 2-broadcast number} of $G$ is $$\gamma_{\stackrel{}{B_2}}(G)=\min\{ \omega(f)\colon f \text{ is a dominating 2-broadcast on } G\}.$$
Moreover, a dominating $2$-broadcast with cost $\gamma_{\stackrel{}{B_2}}(G)$ is called \emph{optimal}.

It is clear from the definition that any dominating set $S$ in a graph $G$ with cardinality $\gamma(G)$ leads to a dominating $2$-broadcast $f$ where $f(v)=1$ if $v\in S$ and $f(u)=0$ for any other vertex. The cost of such $2$-broadcast is $\omega(f)=\gamma(G)$. On the other hand, an optimal dominating $2$-broadcast is also a dominating broadcast. All these relationships can be summarized with the following inequalities:

$$\gamma_{\stackrel{}{B}}(G)\leq \gamma_{\stackrel{}{B_2}}(G) \leq \gamma (G)$$

In general, parameters $\gamma_{\stackrel{}{B}}, \gamma_{\stackrel{}{B_2}}$ and $\gamma$ are different, as shown in the example of Figure~\ref{fig:differentparameters} where $\gamma_{\stackrel{}{B}}(G)=3$, $\gamma_{\stackrel{}{B_2}}(G)=4$ and $\gamma (G)=5$ (circled vertices has non-zero image in an optimal broadcast-type function, these images are also shown).

\begin{figure}[h]
 \captionsetup[subfigure]{labelformat=empty}
 \begin{center}
    \subfigure[$\gamma_{\stackrel{}{B}}(G)=3$]{\includegraphics[width=0.22\textwidth]{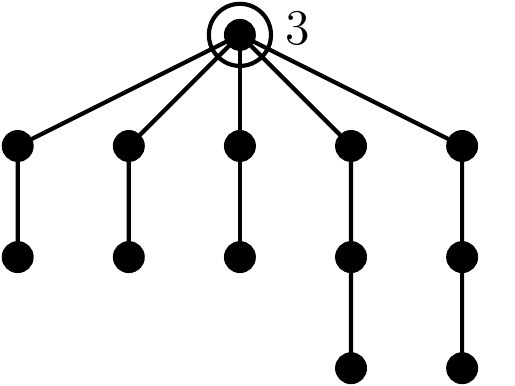}}\hspace{1cm}
    \subfigure[$\gamma_{\stackrel{}{B_2}}(G)=4$]{\includegraphics[width=0.22\textwidth]{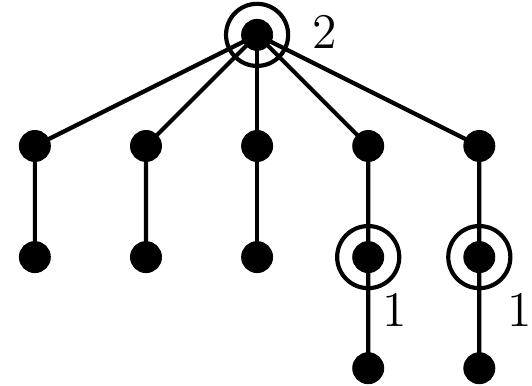}}\hspace{1cm}
    \subfigure[$\gamma (G)=5$]{\includegraphics[width=0.22\textwidth]{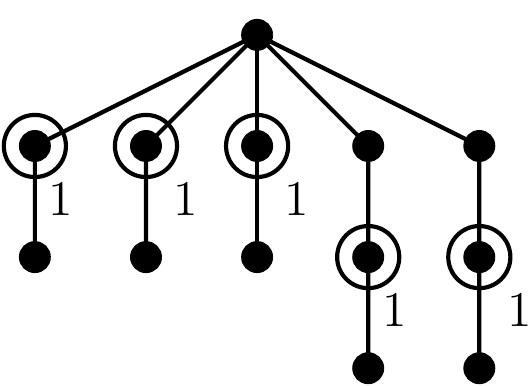}}
\end{center}
\caption{Parameters $\gamma_{\stackrel{}{B}}, \gamma_{\stackrel{}{B_2}}$ and $\gamma$ are different.}
\label{fig:differentparameters}
\end{figure}

The radius is an upper bound for the dominating broadcast number $\gamma_{B}$. Although one might think that the radius plays a similar role in comparison with $\gamma_{\stackrel{}{B_2}}$, Figure~\ref{fig:radius} shows that there is no relationship between these two parameters.

\begin{figure}[ht]
\captionsetup[subfigure]{labelformat=empty}
  \begin{center}
    \subfigure[$\!\!rad(G)\!=\!3\!<\!\gamma_{\stackrel{}{B_2}}(G)\!=\!4$]{\includegraphics[width=0.275\textwidth]{different_2_broadcast}}
    \hspace{0.7cm}
    \subfigure[$\gamma_{\stackrel{}{B_2}}(G)=6<rad(G)=7$]{\includegraphics[width=0.54\textwidth]{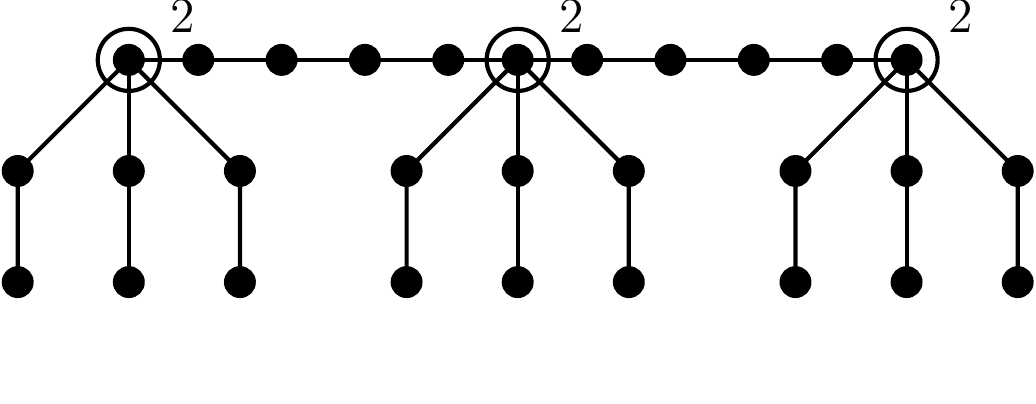}}
    \caption{$\gamma_{\stackrel{}{B_2}}$ is independent from the radius.}
    \label{fig:radius}
\end{center}

\end{figure}

We next present two general properties of parameter $\gamma_{\stackrel{}{B_2}}$ that will be useful in the rest of the paper.
\par\bigskip

\begin{proposition}\label{prop:general} Let $G$ be a graph.
\begin{enumerate}
\item If $e$ is a cut-edge of $G$ and
$G_1,G_2$ are   the connected components of $G-e$, then
$\gamma_{\stackrel{}{B_2}}(G)\leq \gamma_{\stackrel{}{B_2}}(G_1)+\gamma_{\stackrel{}{B_2}}(G_2)$.

\item There exists an optimal dominating $2$-broadcast $f$ such that $f(u)=0$, for every leaf $u$ of $G$.
\end{enumerate}
\end{proposition}
\begin{proof}

\begin{enumerate}

\item Let $f_1,f_2$ optimal dominating $2$-broadcast on $G_1$ and $G_2$, respectively.
Then, the function $f\colon V(G)\to \{0,1,2\}$ such that $f(v)=f_i(v)$ for any $v\in V(G_i)$, is a dominating $2$-broadcast on $G$ with cost $\omega(f)=\sum_{v\in V_{\!\!f}^+} f(v)=\sum_{v\in V^+_{\!\! f_1}} f_1(v)+\sum_{v\in V_{\!\! f_2}^+} f_2(v)= \gamma_{\stackrel{}{B_2}}(G_1)+\gamma_{\stackrel{}{B_2}}(G_2)$.

\item Suppose that $f$ is an optimal dominating $2$-broadcast on $G$ that assigns a positive value to $r$ leaves. Assume that $f(u)>0$ for a leaf $u$ with support vertex $v$. In such a case, the optimality of $f$ implies that $f(v)=0$. Consider the function $g\colon V(G)\to \{0,1,2\}$ satisfying $g(u)=0$, $g(v)=f(u)$, and $g(w)=f(w)$ if $w\neq u,v$. It is clear that $g$ is a dominating $2$-broadcast with cost {$\omega(g)= \omega(f)$, so $g$ is also optimal and has $r-1$ leaves with positive value. Repeating this procedure as many times as necessary, we obtain an optimal $2$-broadcast that assigns the value $0$ to all leaves.} \end{enumerate}
\end{proof}

In \cite{Herke07}, the authors obtained the following result that uses spanning trees as an essential tool to compute the dominating broadcast number of a graph.

\begin{theorem}\cite{Herke07}\label{thm:herke}
Let $G$ be a graph. Then,
$$\gamma_{\stackrel{}{B}}(G)=\min \{ \gamma_{\stackrel{}{B}}(T)\colon T \text{ is a spanning tree of } G\}.$$
\end{theorem}

The proof of this result essentially uses the existence of an \emph{efficient} optimal dominating broadcast \cite{DEHHH06} in every graph, that is, a dominating broadcast $f$ with minimum cost such that any vertex $u$ in $G$ is $f$-dominated by exactly one vertex $v$ with $f(v)>0$.
Nevertheless, there is no similar property for dominating $2$-broadcasts in general.
For instance, the cycle $C_7$ satisfies $\gamma_{\stackrel{}{B_2}}(C_7)=3$ and however, it has no {efficient} optimal dominating $2$-broadcast with cost equal to $3$ (see Figure~\ref{ex:c7}).
Despite this fact, we can get a result similar to that of Theorem~\ref{thm:herke}, by means of an specific construction.

\begin{figure}[htbp]
\begin{center}
\includegraphics[width=0.45\textwidth]{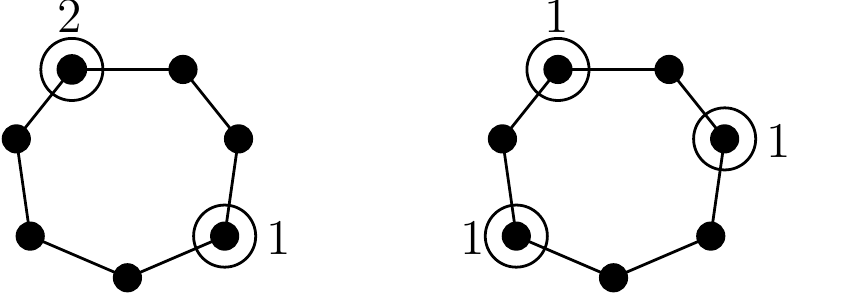}
\caption{There are {exactly} two non-isomorphic optimal dominating {$2$-broadcasts on} the cycle $C_7$.}
\label{ex:c7}
\end{center}
\end{figure}

\begin{theorem}\label{thm:spanning}
Let $G$ be a graph. Then,
$$\gamma_{\stackrel{}{B_2}}(G)=\min \{ \gamma_{\stackrel{}{B_2}}(T)\colon T \text{ is a spanning tree of } G\}.$$
\end{theorem}

\begin{proof}

Let $T$ be a spanning tree of $G$ such that
$\gamma_{\stackrel{}{B_2}}(T)=t=\min \{ \gamma_{\stackrel{}{B_2}}(T)\colon T \text{ is a}$ $\text{spanning tree of } G\}$
and let $f\colon V(T)\to \{0,1,2\}$ be an optimal dominating $2$-broadcast on $T$.
Then, $f$ is also a dominating $2$-broadcast on $G$ and thus $\gamma_{\stackrel{}{B_2}}(G)\leq t$.

Conversely, let $g\colon V(G)\to \{0,1,2\}$ be an optimal dominating $2$-broadcast on $G$. Denote by $V_g^+=\{v_1,\dots ,v_m\}$, with the property $2 \geq g(v_1)\geq g(v_2)\geq \dots \geq g(v_m)\geq 1$.  Define $L(v_1)=\{v_1\}\cup (N(v_1)\cap V_g^0)$ and for every $i\in \{2,\dots , m\}$,
$$L(v_i)=\big(\{v_i\}\cup (N(v_i)\cap V_g^0)\big) \setminus \bigcup\limits_{ j<i} L(v_i).$$
Note that $\bigcup\limits_{i=1}^m L(v_i)$ consists of all vertices in $V_g^+$ and all their neighbors $u$ such that $g(u)=0$. If
$v \in V_g^0\setminus \bigcup\limits_{i=1}^m L(v_i)$, then there exists $w\in V_g^+$ with $g(w)=2$ and $d(w,v)=2$. Moreover, there exists $u\in V_g^0$ such that $d(w,u)=d(u,v)=1$, as $v$ is not a neighbor of any vertex $z$ with $g(z)=1$. By construction, vertex $u$ belongs to exactly one $L(v_j)$ with $g(v_j)=2$.
For each $v\in V_g^0\setminus \bigcup\limits_{i=1}^m L(v_i)$, choose the unique $v_j$ satisfying these conditions and include $v$ in $L(v_j)$.

After this process, we obtain a partition $\{ L(v_1), \dots L(v_m)\}$ of $V(G)$ such that the subgraph induced by $L(v_i)$ is connected for every $i\in \{1,\dots, m\}$. Let $T_i$ the tree rooted in $v_i$ with vertex set $L(v_i)$, obtained by keeping a minimal set of edges of $G$ ensuring that $d_{T_i}(v_i,x)=d_G(v_i,x)$, for every $x\in L(v_i)$, and deleting the rest of edges. It is possible to construct a spanning tree of $G$ by adding
some edges of $G$ to the forest $T_1,T_2,\dots T_m$  in order to get a connected graph $T$ with no cycles.
The property $d_{T_i}(v_i, x)=d_G(v_i, x)\leq g(v_i)$ for every $x\in V(T_i)$, ensures that $g\colon V(T)\to \{0,1,2\}$ is a dominating $2$-broadcast on the spanning tree $T$, so $t\leq \gamma_{\stackrel{}{B_2}}(T)\leq \omega (g)=\gamma_{\stackrel{}{B_2}}(G)$.
\end{proof}

\vspace{1.5cc}
\begin{center}
{\bf 3. COMPUTATIONAL COMPLEXITY}
\end{center}

We now focus on the computational point of view of the calculation of the parameter $\gamma_{\stackrel{}{B_2}}$. The classical {\sc dominating set problem} is a well-known NP-complete decision problem~\cite{GJ79}, in the same way that many others domination related problems. On the other hand, a polynomial algorithm to compute an optimal broadcast domination function of a graph $G$ was quite surprisingly obtained in~\cite{HL06}. So the question arises with following decision problem, is it polynomial or does it behave like most domination-type parameters?

\par\bigskip

\fbox{
\begin{minipage}{10.5cm}
{\sc dominating $2$-broadcast problem}\\
{\sc instance}: A graph $G$ of order $n$ and an integer $k\geq 2$.\\
{\sc question}: Does $G$ have a dominating $2$-broadcast with cost $\leq k$?
\end{minipage}
}

\par\bigskip

We next show that this decision problem is NP-complete for general graphs, by using a reduction of {\sc 3-sat problem}. Therefore, it is a similar situation to the classical dominating number.

\par\bigskip

\begin{theorem}
{\sc dominating $2$-broadcast problem} is NP-complete.
\end{theorem}

\begin{proof}
Since verifying that a given function $f\colon V(G)\to \{0,1,2\}$ is a dominating $2$-broadcast of $G$ can be done in polynomial time, {\sc dominating $2$-broadcast problem} is NP. We will show that it is NP-complete by using a reduction from \textsc{3-sat problem}, quite similar than the one needed for the \textsc{dominating set problem} \cite{HHS98}.

Given an instance $C$ of {\sc 3-sat}, with set of variables $U=\{u_1,\dots, u_n\}$ and clauses $C=\{ C_1, \dots ,C_m\}$, we construct an instance $G(C)$ of {\sc dominating $2$-broadcast problem} as follows. For each variable $u_i$, we consider the gadget $G_i$ in Figure~\ref{fig:gadget}. For each clause $C_j=\{u_k, u_l, u_r\}$, we create a pair of vertices $C_j, \widehat{C_j}$ and we add edges $C_j\widehat{C_j}, u_k\widehat{C_j}, u_l\widehat{C_j}, u_r\widehat{C_j}$ (see Figure~\ref{fig:clause}). Therefore, the graph $G(C)$ has $6n+2m$ vertices and $6n+4m$ edges and it is certainly constructible from the instance $C$ of {\sc 3-sat} in polynomial time.

\begin{figure}[h]
\begin{center}
\includegraphics [width=0.15\textwidth]{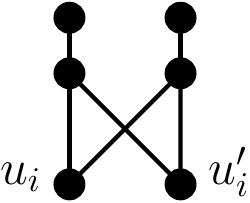}
\caption{Gadget associated to variable $u_i$. }\label{fig:gadget}
\end{center}
\end{figure}

\begin{figure}[h]
\begin{center}
\includegraphics [width=0.7\textwidth]{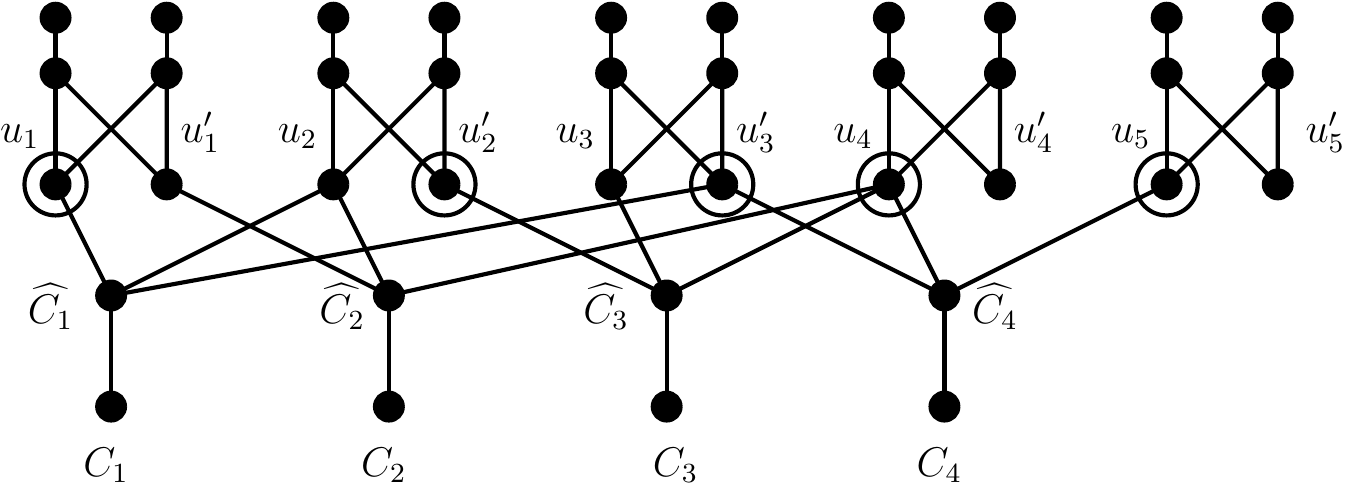}
\caption{A dominating $2$-broadcast in $G(C)$, with $f(v)=2$ for every circled vertex and $f(u)=0$ otherwise.}\label{fig:clause}
\end{center}
\end{figure}

Next we show that $C$ has a satisfying truth assignment if and only if the
graph $G(C)$ has a dominating $2$-broadcast with cost at most $2n$. Firstly, suppose that $C$ has a satisfying truth assignment. Then, the function $f\colon V(G(C))\to \{0,1,2\}$ such that $f(u_i)=2, f(u'_i)=0$ if $u_i$ is true, $f(u_i)=0, f(u'_i)=2$ if $u_i$ is false and $f(x)=0$ if $x\neq u_i, u'_i$, is a dominating $2$-broadcast with cost $\omega(f)=2n$.

Conversely, assume that $G(C)$ has a dominating $2$-broadcast $f$ with cost $\omega(f)\leq 2n$. We must show that $C$ has a satisfying truth assignment. Firstly, note that the leaves of gadget $G_i$ are not $f$-dominated by any vertex outside $V(G_i)$, so either $f$ assigns value $1$ to both support vertices in $V(G_i)$ or $f(u_i)=2$  or $f(u'_i)=2$. Therefore $\omega(f)=2n$ and, in particular, $f(C_j)=0$ for each clause $C_j$. Then $C_j$ is $f$-dominated by a vertex $u$ in a gadget $G_i$. Clearly, $u\in \{u_i,u'_i\}$ and $f(u)=2$. Finally, for each variable $u_i$, assign $u_i$ the value True if $f(u_i)=2$, otherwise assign $u_i$ the value False. It is straightforward to see that this is a satisfying truth assignment for $C$.
\end{proof}

We finish this section by proving that  {\sc dominating $2$-broadcast problem}  can be solved in linear time in the family of trees. Although the linear-time algorithm for trees is not explicitly devised, we prove that the problem is under the conditions to apply the systematic approach to constructing such algorithms suggested in~\cite{blw}.

Let $(T,r)$ be a tree rooted in $r$ and let $f:V(T)\to \{0,1,2\}$. We say that $f$ is a \emph{dominating 2-broadcast function on $T$ except for a subset $S$ of vertices} if any vertex outside $S$ is $f$-dominated and $\forall v\in S, \nexists w\in V(T)$ with $f(w)>0$, such that $d(v,w)\leq f(w)$ (without loss of generality, we say \emph{except for a vertex $u$} when $S$ consists of only one vertex). This allows us to define the  class $\Gamma_f$ as the set of triples $(T,r,f)$ where $(T,r)$ is a rooted tree with root $r$, and $f$ is a dominating 2-broadcast on $T$ except perhaps for a set $S\subseteq V(G)$.

The approach followed in~\cite{blw} is heavily based on the fact that any element in $\Gamma_f$ can be obtained recursively by a rule of composition. The \emph{composition} is defined to be $(T_1, r_1, f_1)\circ (T_2, r_2, f_2)=(T, r_1,f)$ where $V(T)=
V(T_1)\cup V(T_2)$, $E(T)=E(T_1)\cup E(T_2)\cup\{r_1r_2\}$, and the function $f$ is defined as follows:

\par\bigskip

$$
f(v)=\left \{\begin{array}{cc} f_1(v) & \textrm{ if }v\in V(T_1)\\ f_2(v) & \textrm{ if }v\in V(T_2)\end{array} \right.
$$
\par\bigskip

Now $\Gamma_f$ is recursively constructed by using the composition and three primitive elements, namely $(K_1,v,f_0)$, $(K_1,v,f_1)$ and $(K_1,v,f_2)$, where $V(K_1)=\{v\}$ and $f_i(v)=i$ for $i\in\{0,1,2\}$.

We are going to define next a pair $(C,\times)$ where the elements of $C$ are called \emph{classes} (following the notation given in~\cite{blw}) with a multiplicative operation (not necessarily associative nor commutative). Each class captures a certain situation during the construction of a dominating 2-broadcast, and the product of two classes produces the class representing the derived situation.

For our purposes, we distinguish eight different classes
$$C=\{C_1,C_2,C_3,C_4,C_5,C_6,C_7,C_8\}.$$
The triples of $\Gamma_f$ are associated to their classes via the application $h:\Gamma_f\to C$ in the following way:
\begin{itemize}
\item $h((T,r,f))=C_1$ if $f$ is a dominating 2-broadcast function on $T$ and $f(r)=2$.
\item $h((T,r,f))=C_2$ if $f$ is a dominating 2-broadcast function on $T$ and $f(r)=1$.
\item $h((T,r,f))=C_3$ if $f$ is a dominating 2-broadcast function on $T$, $f(r)=0$ and $f(u)=0\ \forall u\in N(r)$ (in particular this implies that $\exists w$ with $f(w)=2$ and $d(r,w)=2$).
\item $h((T,r,f))=C_4$ if $f$ is a dominating 2-broadcast function on $T$, $f(r)=0$, $\exists w\in N(r)$ with $f(w)=1$ and $f(u)\leq 1, \forall u\in N(r)$.
\item $h((T,r,f))=C_5$ if $f$ is a dominating 2-broadcast function on $T$, $f(r)=0$ and $\exists w\in N(r)$ with $f(w)=2$.
\item $h((T,r,f))=C_6$ if $f$ is dominating 2-broadcast function on $T$ except for $S=\{r\}$.
\item $h((T,r,f))=C_7$ if $f$ is a dominating 2-broadcast function on $T$ except for $S\subseteq N[r]$, $\emptyset \neq S\neq \{r\}$.
\item $h((T,r,f))=C_8$ if $f$ is a dominating 2-broadcast function on $T$ except for $S\nsubseteq N[r]$, $S\neq \emptyset$.
\end{itemize}

The multiplicative table of elements in $C$ is given by:
\par\bigskip

$$
\begin{array}{c|cccccccc}
    & C_1 & C_2 & C_3 & C_4 & C_5 & C_6 & C_7 & C_8 \\
 \hline
C_1 & C_1 & C_1 & C_1 & C_1 & C_1 & C_1 & C_1 & C_8 \\
C_2 & C_2 & C_2 & C_2 & C_2 & C_2 & C_2 & C_8 & C_8 \\
C_3 & C_5 & C_4 & C_3 & C_3 & C_3 & C_7 & C_8 & C_8 \\
C_4 & C_5 & C_4 & C_4 & C_4 & C_4 & C_7 & C_8 & C_8 \\
C_5 & C_5 & C_5 & C_5 & C_5 & C_5 & C_5 & C_8 & C_8 \\
C_6 & C_5 & C_4 & C_6 & C_6 & C_3 & C_7 & C_8 & C_8 \\
C_7 & C_5 & C_7 & C_7 & C_7 & C_7 & C_7 & C_8 & C_8 \\
C_8 & C_8 & C_8 & C_8 & C_8 & C_8 & C_8 & C_8 & C_8
\end{array}
$$
\par\bigskip

Following the notation given in~\cite{blw}, those classes containing a primitive element are said to be \emph{primitive classes}, in our case $C_1, C_2$ and $C_6$ and the \emph{accepting classes} are those whose elements verify the desired property, or in our problem, $C_1, C_2, C_3, C_4$ and $C_5$.

Finally, the next result guarantees that $h$ is an homomorphism. The proof is omitted but straightforward.
\begin{theorem}
The function $h\!:\! \Gamma_f\! \to \! C$ defined above is an homomorphism, that is
$$h((T_1,r,f_1)\circ (T_2,r_2,f_2))=h((T_1,r,f_1))\times h((T_2,r_2,f_2))$$
\end{theorem}

According to Corollary 1 in~\cite{blw}, there exists a linear-time algorithm for computing a minimum dominating $f$-broadcast for a tree $T$. Moreover, it is also possible to construct $f$ in linear time.

\vspace{1.5cc}
\begin{center}
{\bf 4. GENERAL UPPER BOUND FOR THE DOMINATING $2$-BROADCAST NUMBER}
\end{center}

Once we have seen that the computation of the limited version of the broadcast dominating number is not an easy problem, in this section a general upper bound on $\gamma_{\stackrel{}{B_2}}$, in terms of the order of the graph, is given. Having in mind Theorem~\ref{thm:spanning} that allows us to obtain the dominating $2$-broadcast number of any graph as the minimum of the parameter of its spanning trees, we focus on bounding $\gamma_{\stackrel{}{B_2}}$ in trees, being also characterized all trees attaining the bound.

First, we prove a technical lemma about the floor and the ceiling functions that will be used later.
Next, we present a family of trees $\mathcal{F}$, that will play an important role in this section. Concretely, we will obtain that a tight upper bound of $\gamma_{\stackrel{}{B_2}}$ in trees of order $n$ is $\lceil 4n/9 \rceil$ and that, essentially, only reach this bound the trees of $\mathcal{F}$.

\begin{lemma}\label{lem.cotafracciogeneral}
Let $a,b,c,d$ be integers. If $a/b\le c/d$, then $a+ \lceil c(n-b)/d \rceil \le \lceil cn/d \rceil$.

\end{lemma}

\begin{proof}
Any pair of real numbers $x$ and $y$ satisfy $\lfloor x-y \rfloor \le \lceil x \rceil - \lceil y \rceil$.
Therefore, $\lfloor bc/d \rfloor =\lfloor cn/d - c(n-b)/d \rfloor \le  \lceil cn/d \rceil - \lceil c(n-b)/d \rceil$, so it is enough to prove that
$a\le \lfloor bc/d \rfloor $. We know that $a$ is an integer such that $a\le  bc/d < \lfloor bc/d \rfloor +1$. Hence, $a\le \lfloor bc/d \rfloor$.
\end{proof}

\begin{figure}[h]
\begin{center}
\includegraphics [width=0.35\textwidth]{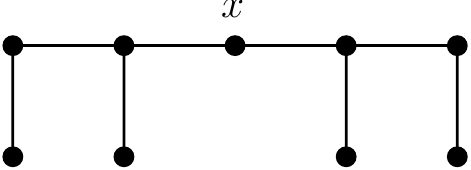}
\caption{$T_9$ has 9 vertices and $\gamma_{\stackrel{}{B_2}} (T_9)=4$. Its unique central vertex is $x$.}
\label{fig:CotaB2Tight nine vertices}
\end{center}
\end{figure}

Let $T_9$ be the tree shown in Figure~\ref{fig:CotaB2Tight nine vertices}, with central vertex $x$. Consider the family $\mathcal{F}$ of trees consisting on $m\geq 1$ copies of $T_9$ in addition to $m-1$ edges, any of them joining central vertices of two copies. Every tree in $\mathcal{F}$ has $9m$ vertices.

\begin{proposition}\label{lem:extremalnueve}
let $T$ be a tree with $n$ vertices such that $T\in \mathcal{F}\cup \{ P_1, P_2, P_4\}$. Then $\gamma_{\stackrel{}{B_2}}(T)=\lceil 4n/9 \rceil$.
\end{proposition}

\begin{proof}
It is straightforward to check that paths $P_1, P_2$ and  $P_4$ satisfies the desired equality. If $T\in \mathcal{F}$, then $T$ has $n=9m$ vertices for some $m\geq 1$. If $m=1$, then $T=T_9$ and the statement is true. Assume that $m\geq 2$ and note that $T$ has $4m$ support vertices that are an optimal dominating set. So, $\gamma_{\stackrel{}{B_2}}(T)\le \gamma (T)=4m$.

Let $f$ be an optimal dominating $2$-broadcast of $T$. Observe that the leaves of a copy of $T_9$ are at distance at least two from the central vertex of the copy, so they do not hear any vertex of another copy. Therefore, $f\vert _{T_9}$, the restriction of $f$ to any copy of $T_9$, is a dominating $2$-broadcast of $T_9$ and $\omega (f\vert _{T_9})\geq 4$. Finally, $4m\le \omega (f)=\gamma_{\stackrel{}{B_2}}(T)\le \lceil 4n/9 \rceil=4m$.
\end{proof}

Next, we obtain a general tight upper bound for trees, being also characterized the family of trees attaining it.

\begin{theorem}\label{pro.cotaB2}
Let $T$ be a tree with $n$ vertices. Then, $\gamma_{\stackrel{}{B_2}} (T)\le \lceil 4n/9 \rceil$.
Moreover, $\gamma_{\stackrel{}{B_2}} (T)=\lceil 4n/9 \rceil$ if and only if $T\in \mathcal{F}\cup \{ P_1, P_2, P_4\}$.

\end{theorem}

\proof
We proceed by induction on the order $n$ of $T$. It is straightforward to check that both statements are true if $1\le |V(T)|\le 4$. Assume that $T$ is a tree of order $n\ge 5$ and that statements are true for all trees of order less than $n$.

If $T$ is  a path, then $\gamma_{\stackrel{}{B_2}}(T)\le \gamma(T)=\lceil n/3 \rceil \le \lceil 4n/9 \rceil.$ On the other hand, $\lceil n/3\rceil=\lceil 4n/9 \rceil$ if and only if $n=1,2,4$. Therefore, there are no paths with at least five vertices satisfying $\gamma_{\stackrel{}{B_2}}(T)\le \gamma(P_n)= \lceil 4n/9 \rceil$. Now, assume that $T$ is not a path.
Let $u,u'$ be antipodal vertices in $T$
and let $u,u_1,\dots ,u_{D-1},u'$ be a shortest path from $u$ to $u'$, where $D=diam (T)$.

Let $v$ be the vertex of degree $d \geq 3$ closer to $u$ lying on this path.
In the following cases, we define trees $T_1$ and $T_2$ by removing a cutting edge of $T$, with orders $n_1$ and $n_2$, respectively. So, by Proposition~\ref{prop:general} and inductive hypothesis, $\gamma_{\stackrel{}{B_2}} (T)\le \gamma_{\stackrel{}{B_2}} (T_1) + \gamma_{\stackrel{}{B_2}} (T_2) \le \lceil 4 n_1 /9 \rceil + \lceil 4 n_2 /9 \rceil$.

\begin{enumerate}

\item
$v=u_1$. Consider $T_1=T(u,u_1u_2)$ and $T_2=T(u',u_1u_2)$ with $n_1\geq 3$ and $n_2\leq n-3$. Applying Lemma~\ref{lem.cotafracciogeneral}, we obtain:
\begin{equation}\label{eq:case1}
\gamma_{\stackrel{}{B_2}} (T)\le 1  + \lceil 4 n_2 /9 \rceil\leq  1  + \lceil 4 (n-3) /9 \rceil \le \lceil 4 n /9 \rceil
 \end{equation}

In this case, there is no tree reaching the bound.
Suppose, on the contrary, that $\gamma_{\stackrel{}{B_2}} (T)=\lceil 4n/9 \rceil$. Then, every inequality in Equation~\ref{eq:case1} must be an equality and $\gamma_{\stackrel{}{B_2}} (T_2)=\lceil 4 n_2 /9 \rceil$. Then, by the inductive hypothesis $T_2 \in \mathcal{F}\cup \{ P_1, P_2, P_4\}$. We study each of the possibilities.

If $T_2\in \mathcal{F}$ then, $n_2=9m$ for some $m\geq 1$ and the terms of Equation~\ref{eq:case1} become:

\begin{eqnarray*}
1 + \lceil 4 n_2 /9 \rceil &=& 1 + \lceil 4 (9m) /9 \rceil =4m+1 \\
   \lceil 4 n /9 \rceil & \geq & 4 (n_2+3) /9 \rceil =\lceil 4 (9m+3) /9 \rceil =4m + \lceil 12/9 \rceil =4m+2.
\end{eqnarray*}

If $T_2=P_4$, then $n_2=4$ and  $1 + \lceil 4 n_2 /9 \rceil =1 + \lceil 16/9 \rceil =3, \lceil 4 n /9 \rceil \geq \lceil 4 (4+3) /9 \rceil =4$.

If $T_2=P_2$, then $n_2=2$ and $1 + \lceil 4 n_2 /9 \rceil =1 + \lceil 8/9 \rceil =2, \lceil 4 n /9 \rceil \geq \lceil 4 (2+3) /9 \rceil = 3$.

If $T_2=P_1$, then $n_2=1$ and $1 + \lceil 4 n_2 /9 \rceil =1 + \lceil 4/9 \rceil =2$. So $\lceil 4 n /9 \rceil =2$ and therefore $4\geq n=n_1+1$. This implies that $n_1=3$, $n=4$ and $T$ is a star, then $\gamma_{\stackrel{}{B_2}} (T)=1$.

In all cases we obtain a contradiction.

\item
$v=u_k, k\geq 3$. Repeat the same reasoning with $T_1=T(u,u_2u_3)$ and $T_2=T(u',u_2u_3)$, having $n_1=3$ and $n_2=n-3$ vertices, respectively.
In this case, there is no tree reaching the bound.

\item $v=u_2$. Consider $T_1=T(u,u_2u_3)$  and $T_2=T(u',u_2u_3)$.

\begin{enumerate}
\item  If $n_1\geq 5$, then $n_2\leq n-5$, $ \gamma_{\stackrel{}{B_2}} (T_1)=2$ and proceeding as before we obtain:
\begin{equation}\label{eq:case3}
\gamma_{\stackrel{}{B_2}} (T)\le  2  + \lceil 4 n_2 /9 \rceil \le  2  + \lceil 4 (n-5) /9 \rceil \le \lceil 4 n /9 \rceil
\end{equation}

In this case, there is also no tree reaching the bound.
Suppose, on the contrary, that $\gamma_{\stackrel{}{B_2}} (T)=\lceil 4n/9 \rceil$. Then, every inequality in Equation~\ref{eq:case3} must be an equality and $\gamma_{\stackrel{}{B_2}} (T_2)=\lceil 4 n_2 /9 \rceil$. Then, by the inductive hypothesis, $T_2 \in \mathcal{F}\cup \{ P_1, P_2, P_4\}$.
If $T_2\in \mathcal{F}$, then $n_2=9m$ for some $m\geq 1$ and
\begin{eqnarray*}
 2 + \lceil 4 n_2 /9 \rceil &\!\!\!\!\! = \!\!\!\!\!&2 + \lceil 4 (9m) /9 \rceil =4m+2 \\
\lceil 4 n /9 \rceil &\!\!\!\!\! \geq \!\!\!\!\!&\lceil 4 (n_2+5) /9 \rceil =\lceil 4 (9m+5) /9 \rceil =4m + \lceil 20/9 \rceil =4m+3.
 \end{eqnarray*}

If $T_2=P_k, k=1,2$, then the function $f\colon V(T)\to \{0,1,2\}$ satisfying $f(u_2)=2$ and $f(w)=0$ otherwise, is an optimal dominating $2$-broadcast of $T$ and $\gamma_{\stackrel{}{B_2}} (T)= 2$. On the other hand, $n\geq 6$. So, $\lceil 4 n /9 \rceil \geq \lceil 24/9 \rceil =3$.\\
In the same way, if $T_2=P_4$, then the function $f\colon V(T)\to \{0,1,2\}$ satisfying $f(u_2)=2$, $f(u')=1$ and $f(x)=0$ otherwise, is an optimal dominating $2$-broadcast of $T$ and $\gamma_{\stackrel{}{B_2}} (T)=3$. In this case, $n\geq 9$ so $\lceil 4 n /9 \rceil \geq \lceil 36/9 \rceil =4$.

In all the previous cases we obtain a contradiction.

 \item If $n_1=4$, then there is a unique leaf $w$ hanging from $v$. Consider $T'=T(u,u_3u_4)$ and $T''=T(u',u_3u_4)$, having $n'\geq 5$ and $n''\leq n-5$ vertices, respectively, and the set $S=V(T')\setminus \{ u,u_1,v,w \}$. Observe that every vertex of $S$ are at distance at most 3 from $u_3$. We distinguish the following cases.

\begin{enumerate}%[i)]
\item \emph{$S=\{u_3\}$}.
  In such a case, $n'=5, n''=n-5$ and the function $f$ satisfying $f(v)=2$ and $f(x)=0$ otherwise, is
  a dominating $2$-broadcast function on $T'$. Proceeding as in the previous cases, we obtain:
  \begin{equation}\label{eq:case4i}
  \gamma_{\stackrel{}{B_2}} (T)\le \gamma_{\stackrel{}{B_2}} (T') + \gamma_{\stackrel{}{B_2}} (T'')\le
 2 + \lceil 4 (n-5) /9 \rceil \le \lceil 4 n /9 \rceil
 \end{equation}

If $\gamma_{\stackrel{}{B_2}} (T)=\lceil 4n/9 \rceil$, then every inequality in Equation~\ref{eq:case4i} must be an equality and $\gamma_{\stackrel{}{B_2}} (T'')=\lceil 4 n'' /9 \rceil$. Hence, by the inductive hypothesis $T'' \in \mathcal{F}\cup \{ P_1, P_2, P_4\}$.

If $T''\in \mathcal{F}$, then $n''=9m$ for some $m\geq 1$ and
\begin{eqnarray*}
 2 + \lceil 4 n'' /9 \rceil &\!\!\!\!\!=\!\!\!\!\!& 2 + \lceil 4 (9m) /9 \rceil =4m+2\\
  \lceil 4 n /9 \rceil &\!\!\!\!\!\geq \!\!\!\!\!& \lceil 4 (n''+5) /9 \rceil =\lceil 4 (9m+5) /9 \rceil =4m+3.
  \end{eqnarray*}
If $T''=P_1$, then the function $f\colon V(T)\to \{0,1,2\}$ satisfying $f(u_2)=2$ and $f(x)=0$ otherwise, is an optimal dominating $2$-broadcast of $T$ and $\gamma_{\stackrel{}{B_2}} (T)= 2$. On the other hand, $n\geq 6$ so $\lceil 4 n /9 \rceil \geq \lceil 24/9 \rceil =3$.\\
If $T''=P_2$, then $2 + \lceil 4 n'' /9 \rceil =2 + \lceil 8 /9 \rceil =3$ and $\lceil 4 n /9 \rceil \geq \lceil  4(5+2)/9 \rceil =4$.\\
In the same way, if $T''=P_4$ and $u_4$ is a leaf of $T''$, then the function $f\colon V(T)\to \{0,1,2\}$ satisfying $f(u_2)=2$, $f(u_{D-1})=1$ and $f(x)=0$ otherwise, is an optimal dominating $2$-broadcast of $T$ and $\gamma_{\stackrel{}{B_2}} (T)=3$. In this case, $n=9$ and thus $\lceil 4 n /9 \rceil = \lceil 36/9 \rceil =4$. In all the previous cases we obtain a contradiction, so $\gamma_{\stackrel{}{B_2}} (T)<\lceil 4n/9 \rceil$.

However, if $T''=P_4$ and $u_4$ is not a leaf of $T''$, then $T=T_9$ and therefore $T\in \mathcal{F}$.

\item
\emph{There is at least one vertex in $S$ at distance 1 from $u_3$, but there are no vertices at distance 2.}
  In such a case, $n'\geq 6, n''\leq n-6$ and the function $f$ satisfying $f(v)=2$ and $f(x)=0$ otherwise, is
  a dominating $2$-broadcast function on $T'$.
  Hence, we obtain:
  \begin{equation}\label{eq:case4ii}
  \gamma_{\stackrel{}{B_2}} (T)\le 2 + \lceil 4 n'' /9 \rceil \le  2 + \lceil 4 (n-6) /9 \rceil \le \lceil 4 n /9 \rceil
 \end{equation}

Suppose that $\gamma_{\stackrel{}{B_2}} (T)=\lceil 4n/9 \rceil$. Hence every inequality in Equation~\ref{eq:case4ii} must be an equality and $\gamma_{\stackrel{}{B_2}} (T'')=\lceil 4 n'' /9 \rceil$. So, by the inductive hypothesis $T'' \in \mathcal{F}\cup \{ P_1, P_2, P_4\}$.

If $T''\in \mathcal{F}$, then $n''=9m$ for some $m\geq 1$ and
\begin{eqnarray*}
  2 + \lceil 4 n'' /9 \rceil &=&2 + \lceil 4 (9m) /9 \rceil =4m+2\\
\lceil 4 n /9 \rceil &\geq& \lceil 4 (n''+6) /9 \rceil =\lceil 4 (9m+6) /9 \rceil =4m+3.
\end{eqnarray*}

If $T''=P_1$, then $2 + \lceil 4 n'' /9 \rceil =2 + \lceil 4 /9 \rceil =3$ and $\lceil 4 n /9 \rceil \geq \lceil  4(6+1)/9 \rceil =4$.\\
If $T''=P_2$, then $2 + \lceil 4 n'' /9 \rceil =2 + \lceil 8 /9 \rceil =3$ and $\lceil 4 n /9 \rceil \geq \lceil  4(6+2)/9 \rceil =4$.\\
If $T''=P_4$, then $2 + \lceil 4 n'' /9 \rceil =2 + \lceil 16 /9 \rceil =4$ and $\lceil 4 n /9 \rceil \geq \lceil  4(6+4)/9 \rceil =5$.\\
In all cases we obtain a contradiction, so $\gamma_{\stackrel{}{B_2}} (T)<\lceil 4n/9 \rceil$.

  \item
  \emph{There is at least one vertex in $S$ at distance 2 from $u_3$, but there are no vertices at distance 3.}
  Thus, $n'\geq 7$, $n''\leq n-7$ and the function satisfying $f(u)=1$, $f(u_3)=2$ and $f(x)=0$ otherwise, is a dominating $2$-broadcast on $T'$. We obtain:

\begin{equation}\label{eq:case4iii}
\gamma_{\stackrel{}{B_2}} (T)\le  3 + \lceil 4 (n-7) /9 \rceil \le \lceil 4 n /9 \rceil
 \end{equation}

Suppose that $\gamma_{\stackrel{}{B_2}} (T)=\lceil 4n/9 \rceil$. Hence,  every inequality in Equation~\ref{eq:case4iii} must be an equality and $\gamma_{\stackrel{}{B_2}} (T'')=\lceil 4 n'' /9 \rceil$. Thus, by the inductive hypothesis $T'' \in \mathcal{F}\cup \{ P_1, P_2, P_4\}$.

If $T''\in \mathcal{F}$, then $n''=9m$ for some $m\geq 1$ and
\begin{eqnarray*}
  3 + \lceil 4 n'' /9 \rceil &=& 3 + \lceil 4 (9m) /9 \rceil =4m+3\\
  \lceil 4 n /9 \rceil &\geq&  \lceil 4 (n''+7) /9 \rceil =\lceil 4 (9m+7) /9 \rceil =4m+4.
\end{eqnarray*}
If $H=P_k, k=1,2$, then the function $f\colon V(T)\to \{0,1,2\}$ satisfying $f(u_1)=1, f(u_3)=2$ and $f(x)=0$ otherwise, is an optimal dominating $2$-broadcast of $T$ and $\gamma_{\stackrel{}{B_2}} (T)=3$. On the other hand, $n\geq 8$ so $\lceil 4 n /9 \rceil \geq \lceil 32/9 \rceil =4$.\\
In the same way, if $T=P_4$, then the function $f\colon V(T)\to \{0,1,2\}$ satisfying $f(u_1)=1$, $f(u_3)=2, f(u')=1$ and $f(x)=0$ otherwise, is an optimal dominating $2$-broadcast of $T$ and $\gamma_{\stackrel{}{B_2}} (T)=4$. In this case, $n\geq 11$ so $\lceil 4 n /9 \rceil \geq \lceil 44/9 \rceil =5$.

In all cases we obtain a contradiction, so $\gamma_{\stackrel{}{B_2}} (T)<\lceil 4n/9 \rceil$.
 \item
 \emph{There is at least one vertex in $S$ at distance 3 from $u_3$.}
  Consider a vertex $t$ at distance 3 from $u_3$ and let $u_3,t_1,t_2,t$ be the $(u_3,t)$-path in $T'$.
  Consider the connected component $C_t=T'(t,u_3t_1)$ of $T_1$.
  We may assume that $C_t$ has exactly 4 vertices,
  as otherwise we can proceed as in the preceding cases because  $t$ and $u'$ are antipodal and we are done.
Therefore, it only remains to prove that the result holds when, for every vertex $t$ of $S$ at distance 3 from $u_3$,  $C_t$ has exactly 4 vertices. In such a case, the tree induced by $C_t$ must be a path $t',t_1,t_2,t$.
Consider the function $f$ defined on $V(T')$ such that $f(u_3)=2$,  $f(u)=1$, $f(t)=1$, if $d(t,u_3)=3$, and $f(x)=0$ otherwise.
On one hand, $f$ is a dominating 2-broadcast on $T'$ with cost $r+3$, where $r$ is the number of vertices in $S$ at distance 3 from $u_3$.
On the other hand, $T'$ has at least $4r+5$ vertices (see Figure~\ref{fig:CotaB2Caseiii}).
\begin{figure}[h]
\begin{center}
\includegraphics [width=0.35\textwidth]{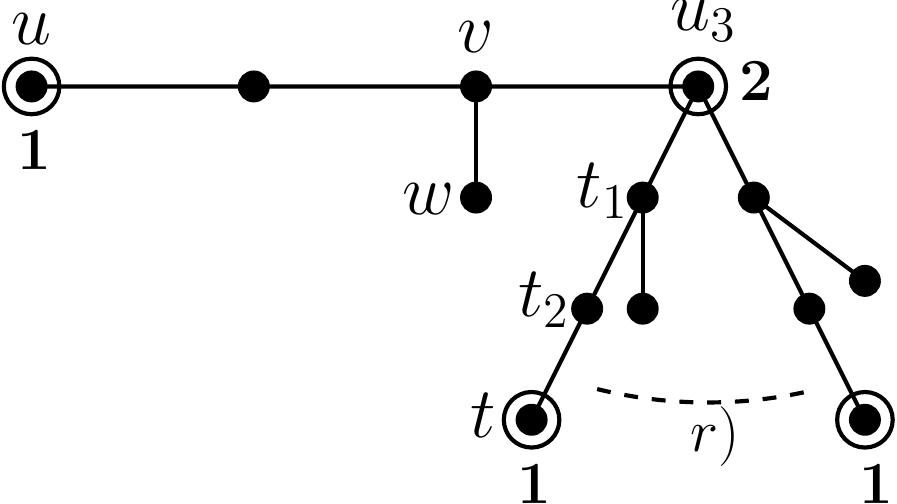}
\caption{$T'$ satisfies $n\ge 4r+5$ and $\gamma_{\stackrel{}{B_2}} (T')\le 3+r$.}
\label{fig:CotaB2Caseiii}
\end{center}
\end{figure}

It is straightforward to check that $r\ge 1$ implies that $\displaystyle{\frac{r+3}{4r+5}\le \frac 49}$.
Therefore, we have:
\begin{equation}\label{eq:case4iv}
\gamma_{\stackrel{}{B_2}} (T)\le  (r+3) + \lceil 4 (n-(4r+5)) /9 \rceil \le \lceil 4 n /9 \rceil
 \end{equation}

Suppose that $\gamma_{\stackrel{}{B_2}} (T)=\lceil 4n/9 \rceil$, then every inequality in Equation~\ref{eq:case4iv} must be an equality and $\gamma_{\stackrel{}{B_2}} (T'')=\lceil 4 n'' /9 \rceil$. By the inductive hypothesis, $T'' \in \mathcal{F}\cup \{ P_1, P_2, P_4\}$.

If $T''=P_k, k=1,2$, then the function $f$ defined on $V(T)$ such that $f(u_3)=2$, $f(u)=1$, $f(t)=1$, if $d(t,u_3)=3$, and $f(x)=0$ otherwise is a dominating $2$-broadcast on $T$ with cost $r+3$, thus $\gamma_{\stackrel{}{B_2}} (T)\leq r+3$. However, $\gamma_{\stackrel{}{B_2}} (T') + \gamma_{\stackrel{}{B_2}} (T'')=(r+3)+1=r+4$.\\
Similarly, if  $T''=P_4$, then the function $f\colon V(T)\to \{0,1,2\}$ satisfying $f(u_3)=2$, $f(u)=1$, $f(t)=1$, if $d(t,u_3)=3$, $f(u')=1$ and $f(x)=0$ otherwise, is a dominating $2$-broadcast on $T$ with cost $r+4$, thus $\gamma_{\stackrel{}{B_2}} (T)\leq r+4$. However, $\gamma_{\stackrel{}{B_2}} (T') + \gamma_{\stackrel{}{B_2}} (T'')=(r+3)+2=r+5$.\\
In each case we obtain a contradiction, so $\gamma_{\stackrel{}{B_2}} (T)<\lceil 4n/9 \rceil$.

Finally, if $T''\in \mathcal{F}$, then $n''=9m$ for some $m\geq 1$ and
$$(r+3) + \lceil 4n'' /9 \rceil =(r+3) + \lceil 4 (9m) /9 \rceil =4m+(r+3)$$
$$\lceil 4 n /9 \rceil \geq \lceil 4 (9m+(4r+5)) /9 \rceil =4m + \lceil (16r+20)/9 \rceil$$
Therefore, $4m+(r+3)=4m + \lceil (16r+20)/9 \rceil$. Thus, $r+3=\lceil (16r+20)/9 \rceil=\lceil (9r+18)/9 + (7r+2)/9\rceil= r+2+\lceil(7r+2)/9\rceil$, so $1=\lceil(7r+2)/9\rceil$ or equivalently $r=1$. Note that this means that $T'=T_9$ and $u_3$ is its center. So, in order to prove that $T\in \mathcal{F}$ we just need to ensure that $u_4$ is the central vertex of its copy of $T_9$.

Let $H_9$ be the copy of $T_9$ in $T''$ containing $u_4$ and assume on the contrary that $u_4$ is not the central vertex of $H_9$. If $u_4$ is a leaf, then we denote by $z$ its support vertex. If $u_4$ is a support vertex, then let $z=u_4$. Define the function $f\colon V(T)\to \{0,1,2\}$ in the following way: $f(u_3)=2$, $f(u)=1$, $f(t)=1$ if $t\in V(T')$ and $d(t,u_3)=3$, $f(y)=1$ if $y$ is a support vertex of $T''$ other than $z$, and $f(x)=0$ otherwise. Thus, $f$ is a dominating $2$-broadcast of $T$ with cost $4m-1+4=4(m+1)-1=(4(9m+9)/9)-1=4n/9-1$. A contradiction with $\gamma_{\stackrel{}{B_2}} (T)=\lceil 4n/9 \rceil$. \qed
\end{enumerate}
\end{enumerate}
\end{enumerate}

As a consequence of Theorem \ref{thm:spanning} and Theorem \ref{pro.cotaB2} we conclude the following result, that extend the upper bound that we found for trees to any graph.

\begin{corollary}\label{cor:general graphs}
Let $G$ be a graph of order $n$, then $$\gamma_{\stackrel{}{B_2}} (G)\le \lceil 4n/9 \rceil .$$
\end{corollary}

\begin{remark}
As we pointed out in Section 2, there is a general relationship of among the dominating $2$-broadcast, the dominating broadcast and the classical domination numbers, summarized the inequality chain
$$\gamma_{\stackrel{}{B}}(G)\leq \gamma_{\stackrel{}{B_2}}(G) \leq \gamma (G).$$
The upper bound found in Corollary~\ref{cor:general graphs} has a similar flavour to both known upper bounds for the other parameters:
$\gamma_{\stackrel{}{B}}(G)\leq \lceil \frac{n}{3}\rceil$ (\cite{Herke07}) and $\gamma (G)\leq \lfloor \frac{n}{2} \rfloor$ (\cite{Ore62}).
\end{remark}

\vspace{1.5cc}
\begin{center}
{\bf 5. DOMINATING PATHS}
\end{center}

In the previous section we have obtained an upper bound for the dominating $2$-broadcast number and we have showed that it is tight. However, we have also seen that the family of trees that reach this limit is very specific, suggesting that the graphs that have maximum dominating $2$-broadcast number are themselves of a very specific type, since they must have every spanning tree in such a restricted tree family. This leads us to think that the dominating $2$-broadcast number of graphs with spanning trees of some other particular type, could be limited by a smaller upper bound. Specifically we look for trees that are simple so that the parameter is small, but at the same time, have a rich enough structure so that a wide variety of graphs have them like spanning trees. For this we have chosen caterpillars.

A caterpillar is a tree in which the removal of all leaves yields a path. Graphs having a caterpillar as a spanning tree are said to have a vertex dominating path, that is a path $P$ such that every vertex outside $P$ has a neighbor on $P$. This particular type of spanning trees could be considered as a generalization of Hamiltonian paths (see \cite{HaNa65,Ve83}) and graphs having such extremal spanning trees have been studied. For instance, it is known that sufficient conditions for having a spanning path are related to large enough minimum degree (see \cite{Bro88, Di52}). With these graphs in mind, we have studied the behavior of the dominating $2$-broadcast in caterpillars and we have found a smaller tight upper bound for these particular trees.

\begin{proposition}\label{pro.cotaB2caterpillar}
Let $T$ be a caterpillar of order $n\ge 1$. Then, $\gamma_{\stackrel{}{B_2}}(T)\le \lceil 2n/5 \rceil$.
\end{proposition}

%\begin{proof}
\proof
We proceed by induction on the order of the caterpillar. By inspection of all cases, the result is true for caterpillars of order at most 6  (see Figure~\ref{fig:CaterpillarsOrden6}).

\begin{figure}[h]
\begin{center}
\includegraphics [width=1\textwidth]{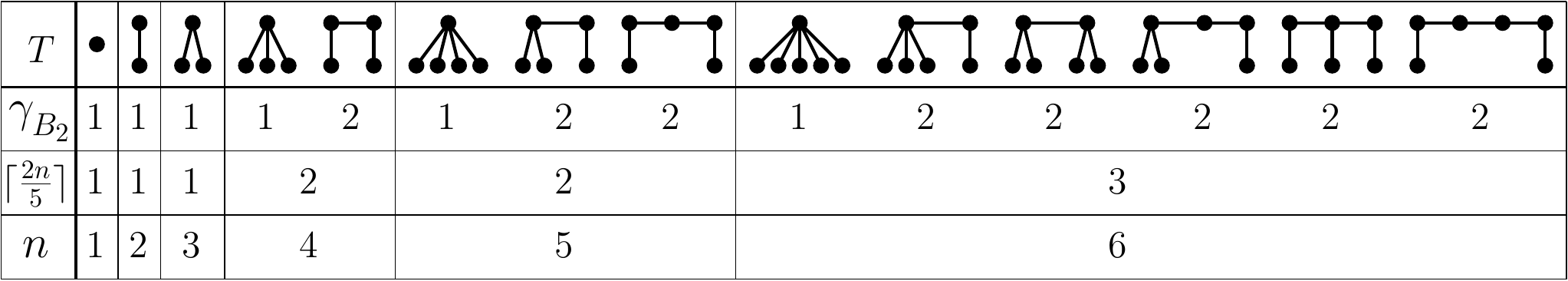}
\caption{All caterpillars $T$ of order $n\le 6$ satisfy $\gamma_{\stackrel{}{B_2}} (T) \le \big\lceil 2 n /5 \big\rceil$. }\label{fig:CaterpillarsOrden6}
\end{center}
\end{figure}

Let $T$ be a caterpillar of order $n\ge 7$ and assume that the statement is true for caterpillars of order less than $n$.
If $T$ has a support vertex $u$ with at least two leaves $v_1, v_2$ hanging from it, then the caterpillar $T^*$ obtained from $T$ by removing vertex $v_1$ has order $n-1<n$, it clearly satisfies $\gamma_{\stackrel{}{B_2}} (T)=\gamma_{\stackrel{}{B_2}} (T^*)$ and by the inductive hypothesis
$$\gamma_{\stackrel{}{B_2}} (T)=\gamma_{\stackrel{}{B_2}} (T^*)\le \lceil 2n^*/5 \rceil \le \lceil 2n/5 \rceil.$$
If every support vertex of $T$ has exactly one leaf hanging from it, then the path $u_1...u_r$ obtained by removing all leaves from $T$ has order at least 4.
Observe that there is a leaf hanging from $u_1$ and another leaf hanging from $u_r$.
If there is no leaf hanging from $u_2$, {then} consider the tree $T_1=T(u_2,u_2u_3)$, which has order 3, and the tree $T_2=T(u_3,u_2u_3)$. By  the inductive hypothesis, Proposition~\ref{prop:general} and Lemma~\ref{lem.cotafracciogeneral}, we have:
$$\gamma_{\stackrel{}{B_2}} (T)\le \gamma_{\stackrel{}{B_2}} (T_1) + \gamma_{\stackrel{}{B_2}} (T_2)\le
 1 + \lceil 2 (n-3) /5 \rceil \le \lceil 2 n /5 \rceil .$$
If there is a leaf hanging from $u_2$, {then} consider the tree $T_1=T(u_3,u_3u_4)$, which has order 5 or 6, and the tree $T_2=T(u_4,u_3u_4)$. The function $f$ such that $f(u_2)=2$ and $f(x)=0$, if $x\not= u_2$, is a dominating 2-broadcast on {$T_1$.}
By the inductive hypothesis, Proposition~\ref{prop:general} and Lemma~\ref{lem.cotafracciogeneral}, we have:
\begin{center}
$\hspace{2cm} \gamma_{\stackrel{}{B_2}} (T)\le \gamma_{\stackrel{}{B_2}} (T_1) + \gamma_{\stackrel{}{B_2}} (T_2)\le
 2 + \lceil 2 (n-5) /5 \rceil \le \lceil 2 n /5 \rceil .$\qed
 \end{center}
%\end{proof}

Having a spanning caterpillar is equivalent to having a dominating path and, as we have already mentioned, sufficient conditions are known so that a graph contains a dominant path.  Some of them are related to the fact that the minimum degree of the graph is large enough. As an example, we quote this result from \cite{FGJW17}.

\begin{theorem}\cite{FGJW17}
For $n\geq 2$, every connected $n$-vertex graph $G$ with $\delta(G) > \frac{n-1}{3}-1$ has a dominating path, and the inequality is sharp.
\end{theorem}

Finally, Proposition~\ref{pro.cotaB2caterpillar} and Theorem~\ref{thm:spanning} provide an upper bound of $\gamma_{\stackrel{}{B_2}}$, for graphs having a spanning path.
\begin{corollary}
Let $G$ be a graph of order $n$ and having a spanning path. Then, $\gamma_{\stackrel{}{B_2}}(G)\leq\lceil 2n/5\rceil.$
\end{corollary}

\vspace{0.5cc}

\noindent {\bf Acknowledgments.}\\
Partially supported by projects MTM2014-60127-P(MINECO), MTM2015-63791-R (MINECO/FEDER), Gen.Cat. DGR2014SGR46 and Junta de Andaluc\'ia FQM305. %H2020-MSCA-RISE project 734922-CONNECT
\vspace{-0.9cm}
\begin{figure}[h!]
\begin{minipage}[l]{0.3\textwidth}
\includegraphics[trim=10cm 6cm 10cm 5cm,clip,scale=0.15]{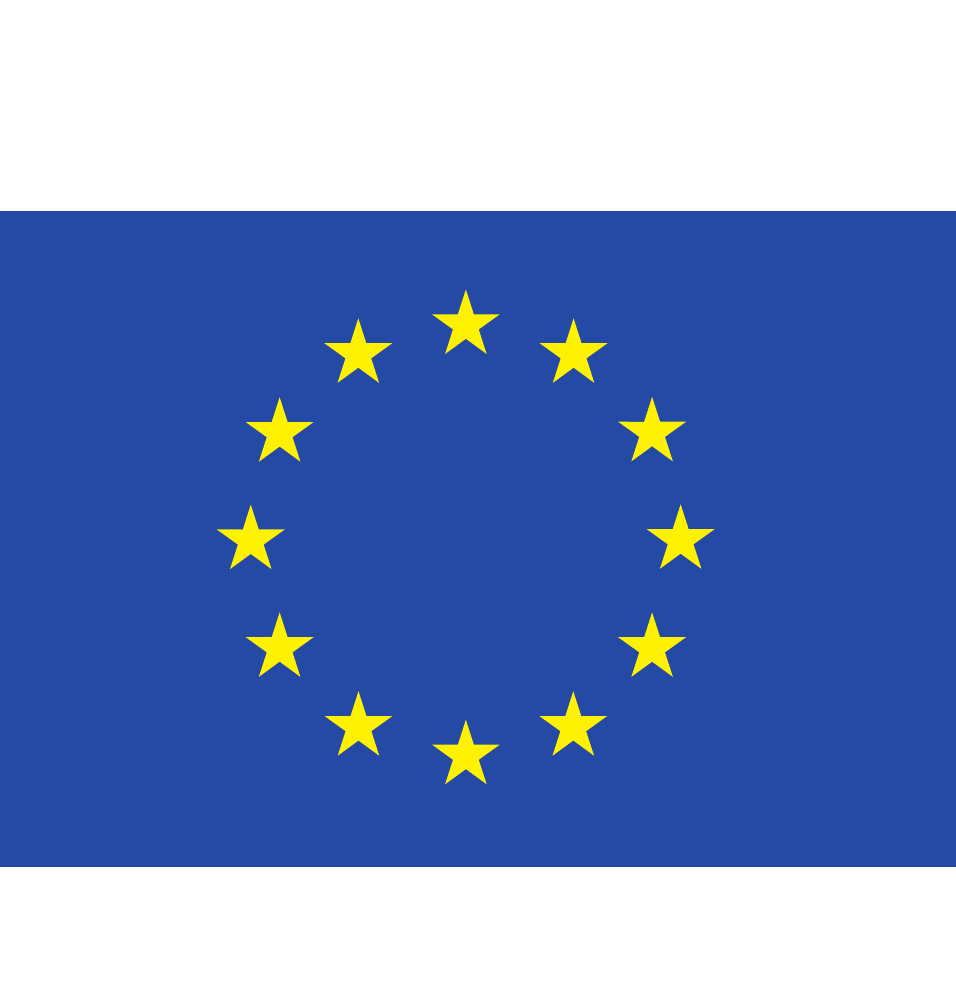}
\end{minipage}\hspace{-2cm}
\begin{minipage}[l][1cm]{0.85\textwidth}
This project has received funding from the European Union's Horizon 2020 research and innovation programme under the Marie Sk\l{}odowska-Curie grant agreement No 734922.
\end{minipage}
\end{figure}


\vspace{-0.6cm}

\vspace{2cc}


\begin{thebibliography}{1}




\bibitem{Berge58} {\small {\sc C. Berge:} {\it  Theory of Graphs and its Applications. Collection Universitaire de Math\'ematiques, vol. 2.} Dunod, Paris, 1958.}

\bibitem{blw} {\small {\sc M.W. Bern, E.L. Lawler, A.L. Wong:} {\it  Linear-time computation of optimal subgraphs of decomposable graphs.} J. Algorithms, {\bf 8} (1987), 216--235.}

\bibitem{Bro88} {\small {\sc H.J. Broersma:} {\it Existence of $\Delta_{\lambda}$-cycles and $\Delta_{\lambda}$-paths.} J. Graph Theory, {\bf 12} (1988), 499--507.}

\bibitem{CLZ11} {\small {\sc G. Chartrand, L. Lesniak, P. Zhang:} {\it  Graphs and Digraphs, (5th edition).} CRC Press, Boca Raton, Florida, 2011.}

\bibitem{CHM11} {\small {\sc E.J. Cockayne, S. Herke, C.M. Mynhardt:} {\it  Broadcasts and domination in trees.} Discrete Math., {\bf 311(13)} (2011), 1235--1246.}

\bibitem{Di52} {\small {\sc G.A. Dirac:} {\it Some theorems on abstract graphs.} Proc. London Math. Soc. {\bf 2} (1952), 69--81.}

\bibitem{DEHHH06} {\small {\sc J.E. Dunbar, D.J. Erwin, T.W. Haynes, S.M. Hedetniemi, S.T. Hedetniemi:} {\it  Broadcast in graphs.} Discret. Appl. Math., {\bf 154} (2006), 59--75.}

\bibitem{Erwin04} {\small {\sc D.J. Erwin:} {\it  Dominating brodcast in graphs.} Bulletin of the ICA, {\bf 42} (2004), 89--105.}

\bibitem{FGJW17} {\small {\sc R.J. Faudree, R.J. Gould, M.S. Jacobson, D.B. West:} {\it Minimum degree and dominating paths.} J. Graph Theory, {\bf 84(2)} (2017), 202--213.}

\bibitem{GJ79}  {\small {\sc M.R. Garey, D.S. Johnson:} {\it  Computers and Intractability: A Guide to the Theory of NP-Com.pleteness.} Freeman, New York, 1979.}

\bibitem{HaNa65} {\small {\sc F. Harary, C. St. J. A. Nash-Williams:} {\it On eulerian and hamiltonian graphs and line graphs.} Canad. Math. Bull., {\bf 8} (1965), 701--709.}

\bibitem{HHS98} {\small {\sc T.W. Haynes, S. Hedetniemi, P. Slater:} {\it  Fundamentals of Domination in Graphs.} CRC Press, 1998.}

\bibitem{HL06} {\small {\sc P. Heggernes, D. Lokshtanov:} {\it  Optimal broadcast domination in polynomial time.} Discrete Math., {\bf 306(24)} (2006), 3267--3280.}

\bibitem{Herke07} {\small {\sc S. Herke:} {\it Dominating Broadcasts in Graphs} Master's Dissertation, University of Victoria, 2009.}

\bibitem{HM09} {\small {\sc S. Herke, C.M. Mynhardt:} {\it  Radial trees.} Discrete Math., {\bf  309(20)} (2009), 5950--5962.}

\bibitem{JK13} {\small {\sc N. Jafari, F. Khosravi:} {\it  Limited dominating broadcast in graphs.} Discrete Math. Algorithms Appl., {\bf 5(4)} (2013), 1350025, 9 pp.}

\bibitem{JMM03} {\small {\sc R.E. Jamison, F.R. McMorris, H.M. Mulder:} {\it  Graphs with only caterpillars as spanning trees.} Discret. Math., {\bf  272} (2003), 81--95.}

\bibitem{Liu68} {\small {\sc C.L. Liu:} {\it Introduction to Combinatorial Mathematics.} McGraw-Hill, New York, (1968).}

\bibitem{MW13} {\small {\sc C.M. Mynhardt, J. Wodlinger:} {\it  A class of trees with equal broadcast and domination numbers.} Australas. J. Combin., {\bf 56} (2013), 3--22.}

\bibitem{Ore62} {\small {\sc O.Ore:} {\it  Theory of Graphs, American Mathematical Society Publication, vol. 38.} American Mathematical Society, Providence, 1962.}

\bibitem{Ve83} {\small {\sc H.J. Veldman:} {\it Existence of dominating cycles and paths.} Discrete Math., {\bf  43} (1983), 281--296.}

\end{thebibliography}
\end{document}